\documentclass[11pt,a4paper]{article}

\usepackage[T1]{fontenc}
\usepackage[utf8]{inputenc}
\usepackage{lmodern}
\usepackage{microtype}

\usepackage{amsmath,amssymb,amsthm,mathtools}

\usepackage{enumitem}
\usepackage{xcolor}
\usepackage{geometry}
\usepackage{aliascnt}

\usepackage{hyperref}
\usepackage[nameinlink,capitalise,noabbrev]{cleveref}

\hypersetup{
  colorlinks=true,
  linkcolor=black,
  citecolor=black,
  urlcolor=black,
  pdftitle={The class SJ is not countably recognizable},
  pdfauthor={M. Brescia, B.G. Di Siena, E. Ingrosso, M. Trombetti}
}

\usepackage{refcheck}

\makeatletter

\AtBeginDocument{%

  \let\RC@cref\cref
  \renewcommand{\cref}[1]{%
    \@for\RC@label:=#1\do{%
      \wrtusdrf{\RC@label}%
    }%
    \RC@cref{#1}%
  }%

  \let\RC@Cref\Cref
  \renewcommand{\Cref}[1]{%
    \@for\RC@label:=#1\do{%
      \wrtusdrf{\RC@label}%
    }%
    \RC@Cref{#1}%
  }%

  \let\RC@crefrange\crefrange
  \renewcommand{\crefrange}[2]{%
    \wrtusdrf{#1}%
    \wrtusdrf{#2}%
    \RC@crefrange{#1}{#2}%
  }%

  \let\RC@Crefrange\Crefrange
  \renewcommand{\Crefrange}[2]{%
    \wrtusdrf{#1}%
    \wrtusdrf{#2}%
    \RC@Crefrange{#1}{#2}%
  }%
}

\makeatother

\newtheorem{theorem}{Theorem}[section]
\newtheorem*{maintheorem}{Main Theorem}

\newaliascnt{lemma}{theorem}
\newtheorem{lemma}[lemma]{Lemma}
\aliascntresetthe{lemma}

\newaliascnt{proposition}{theorem}
\newtheorem{proposition}[proposition]{Proposition}
\aliascntresetthe{proposition}

\newaliascnt{corollary}{theorem}
\newtheorem{corollary}[corollary]{Corollary}
\aliascntresetthe{corollary}

\newaliascnt{claim}{theorem}

\aliascntresetthe{claim}

\theoremstyle{definition}

\newaliascnt{definition}{theorem}
\newtheorem{definition}[definition]{Definition}
\aliascntresetthe{definition}

\newaliascnt{remark}{theorem}

\aliascntresetthe{remark}

\newaliascnt{convention}{theorem}

\aliascntresetthe{convention}

\crefname{theorem}{Theorem}{Theorems}
\Crefname{theorem}{Theorem}{Theorems}

\crefname{lemma}{Lemma}{Lemmas}
\Crefname{lemma}{Lemma}{Lemmas}

\crefname{proposition}{Proposition}{Propositions}
\Crefname{proposition}{Proposition}{Propositions}

\crefname{corollary}{Corollary}{Corollaries}
\Crefname{corollary}{Corollary}{Corollaries}

\crefname{claim}{Claim}{Claims}
\Crefname{claim}{Claim}{Claims}

\crefname{definition}{Definition}{Definitions}
\Crefname{definition}{Definition}{Definitions}

\crefname{remark}{Remark}{Remarks}
\Crefname{remark}{Remark}{Remarks}

\crefname{convention}{Convention}{Conventions}
\Crefname{convention}{Convention}{Conventions}

\newcommand{\SJ}{\operatorname{SJ}}
\newcommand{\Hclass}{\mathcal H}

\newcommand{\N}{\mathbb N}
\newcommand{\defect}{\operatorname{d}}
\newcommand{\sn}{\mathrel{\trianglelefteq\trianglelefteq}}

\title{The group class $\SJ$ is not countably recognizable}
\author{M. Brescia --- B.G. Di Siena --- E. Ingrosso --- M. Trombetti}
\date{}

\begin{document}
\maketitle

\begin{abstract}
A group is an {\it $\SJ$-group} if it admits a general abelian series all of whose terms are subnormal of finite defect in the whole group.  We construct a group $G$ which is not an~\hbox{$\SJ$-group}, although every countable subgroup of $G$ is an $\SJ$-group.
\end{abstract}

\section{Introduction}

A group class $\mathfrak X$ is said to be \emph{countably recognizable} if every group all of whose countable subgroups belong to $\mathfrak X$ is itself an $\mathfrak X$-group. Countable recognizability is therefore a local-to-global principle at the first infinite cardinal: it allows one to detect a property of an arbitrary group, possibly uncountable, from the structure of its countable subgroups. The point is especially significant for classes which are not local in the usual sense. Solubility and nilpotency, for instance, are countably recognizable although neither class is local. More generally, countable recognizability has proved useful whenever the property under consideration is expressed by series, rank conditions, residual conditions, or restrictions on large families of subgroups (see \cite{Baer1962,deGiovanniTrombetti2021,Neumann1973,Phillips1971} and the references therein).

The first examples predate the terminology. \v{C}ernikov proved that hypercentrality can be detected on countable subgroups \cite{Chernikov1950}, and Baer obtained the analogous result for hyperabelian groups \cite{Baer1955}. Baer then initiated the systematic study of countably recognizable group-theoretic properties in \cite{Baer1962}. Subsequent work showed that the phenomenon is not confined to generalized soluble groups. In the finitary setting, Hall, Wehrfritz, and Leinen--Puglisi obtained recognition theorems and related local-system results for large finitary linear and permutation groups \cite{Hall1988,LeinenPuglisi1997,LeinenPuglisi2000,Wehrfritz1997}. These results already illustrate one of the main uses of the subject: a group of arbitrary cardinality can often be controlled by reducing the problem to suitable countable subgroups.

A particularly substantial line of work concerns rank conditions. Dixon, Evans and Smith developed a structural theory for locally soluble-by-finite groups whose proper subgroups, or whose subgroups of infinite rank, satisfy strong finite-rank restrictions; see for example \cite{DixonEvansSmith2000,DixonEvansSmith2005}. In this context they later proved that soluble-by-(finite rank) groups form a countably recognizable class, with analogous results for nilpotent-by-(finite rank) groups \cite{DixonEvansSmith2007}. Smith subsequently extended these methods to classes which are (locally soluble)-by-finite rank or finite rank-by-(locally soluble) \cite{Smith2009}. Thus countable recognizability is closely tied to the general programme of recovering global information from restrictions on subgroups of infinite rank.

There is a parallel and now extensive theory of genuinely uncountable groups in which hypotheses are imposed on subgroups having the same cardinality as the ambient group. Among the problems considered are normality, permutability, transitive normality and pronormality of large subgroups (see \cite{DeFalcoDeGiovanniHeinekenMusella2017,DeFalcoEvansDeGiovanniMusella2018,DeFalcoDeGiovanniMusellaSysak2019,DeFalcoDeGiovanniMusella2021,DeFalcoDeGiovanniMusella2025}). Although these ``large-subgroup'' problems are formally different from countable recognizability, the two methods interact naturally: if a class is countably recognizable, failure of the property is already witnessed inside a countable subgroup, and that witness can often be incorporated into a suitable proper subgroup of full cardinality. An explicit instance occurs in \cite{DeFalcoMusellaZaccardo2025}, whose arguments use the countable recognizability of extended residual finiteness (ERF), established in \cite{deGiovanniTrombetti2018Residual}, to pass between countable subgroups and the uncountable ambient group. This illustrates why countable recognizability is not merely a classification question, but a working tool in the structure theory of large groups.

The property is nevertheless far from automatic. Higman's work on almost free groups yields the classical failure of countable recognizability for free groups \cite{Higman1951}. In a different direction, Hickin and Phillips constructed an uncountable locally finite $p$-group showing that an ascending-series condition may hold on all countable subgroups and fail globally~\cite{HickinPhillips1970}. Finitary linearity also exhibits both positive recognition results in restricted settings and negative phenomena in general; compare \cite{LeinenPuglisi1997,LeinenPuglisi2000,Wehrfritz1997}. These examples show that countable recognizability is a genuine structural property of a class, rather than a formal consequence of its definition.

We consider here the class $\SJ$ of groups possessing a general abelian series whose terms are subnormal, with finite defect, in the ambient group. This is one of the classical classes of generalized soluble groups in the sense of Robinson \cite{RobinsonFGSG2}. The question whether $\SJ$ is countably recognizable was posed explicitly as Problem ADV--2B in 2016 \cite{ADV2016}; it was still singled out in the 2021 survey on countable recognizability as the remaining open problem among the principal generalized soluble classes \cite{deGiovanniTrombetti2021}. The problem also appears in the AGTA Open Problems collection, where a prize of EUR~1000 is offered for a solution \cite{AGTAOpenProblems}. We give a negative answer.

\begin{maintheorem}
There exists a group $G$ such that every countable subgroup of $G$ belongs to~$\SJ$, whereas $G\notin\SJ$. Consequently the class $\SJ$ is not countably recognizable.
\end{maintheorem}

We briefly describe the construction and the two mechanisms behind the proof. If $K\leq G$, write $K^G$ for the normal closure of $K$ in $G$ and put

$$
K^{G,0}=G,\qquad K^{G,1}=K^G,\qquad K^{G,n+1}=K^{\,K^{G,n}}.
$$

Thus $K^{G,n}$ is the least subnormal subgroup of $G$ of defect at most $n$ containing $K$. For every positive integer $r$ we construct an $\SJ$-group $A_r$ containing commuting involutions $x_r,y_r$ such that $\langle x_r,y_r\rangle^{A_r,j}$ is perfect for $1\leq j\leq r$. We then take

$$
E=\bigoplus_{\alpha<\omega_1}\langle g_\alpha\rangle,\qquad |g_\alpha|=2.
$$

For every $\beta<\omega_1$ choose an injection $e_\beta:\beta\to\mathbb N$, and for $\alpha<\beta$ put

$$
n(\alpha,\beta)=e_\beta(\alpha)+1.
$$

For each pair $\alpha<\beta$ we take a copy $A_{\alpha\beta}$ of $A_{n(\alpha,\beta)}$ and identify its two chosen involutions with $g_\alpha$ and $g_\beta$. The resulting group $G$ therefore satisfies that, for every $1\leq j\leq n(\alpha,\beta)$, $\langle g_\alpha,g_\beta\rangle^{A_{\alpha\beta},j}$ is perfect.

Suppose, towards a contradiction, that \(G\in\SJ\), and fix a general abelian series for \(G\). For each \(\alpha<\omega_1\), let \(U_\alpha<V_\alpha\) be the jump of the series containing \(g_\alpha\), so that \(g_\alpha\in V_\alpha\setminus U_\alpha\). Since the terms of such a series have finite subnormal defect in \(G\), where the defect of a subgroup is the length of a finite subnormal chain connecting it to \(G\), uncountably many of the groups \(V_\alpha\) must have the same defect, say \(d\).

The combinatorics built into the construction of \(G\) then allow us to choose \(\alpha<\beta<\omega_1\) among these indices with \(n(\alpha,\beta)>d\). On the other hand, the subgroup generated by \(g_\alpha\) and \(g_\beta\) was chosen so that, after sufficiently many iterated normal closures, one reaches a perfect subgroup. This is incompatible with \(g_\alpha\) and \(g_\beta\) lying in jumps of an abelian series whose relevant terms have uniformly bounded subnormal defect. Hence \(G\notin\SJ\).

The situation is different for countable subgroups. Any countable subgroup of $G$ involves only countably many indices $\alpha<\omega_1$, and is therefore contained in a corresponding group $G_S$, generated by the factors $A_{\alpha\beta}$ with $\alpha,\beta\in S$ together with the elements $g_\alpha$, for some countable set $S\subseteq\omega_1$. After enumerating $S$, one can successively factor out the normal closures of the corresponding generators. The kernels arising in this process decompose as iterated amalgamated free products. By passing through suitable elementary abelian quotients, while keeping control of the corresponding kernels, we obtain that all these successive kernels belong to $\SJ$. The final quotient is a free product of groups in $\SJ$, and hence $G_S\in\SJ$. Consequently every countable subgroup of $G$ belongs to $\SJ$.

\section{Preliminaries}

The aim of this section is to recall some easy facts and basic definitions that are used in the proof of the Main Theorem. We start with the core definition of the paper.

\begin{definition}
A group class $\mathfrak X$ is \emph{countably recognizable} if a group belongs to $\mathfrak X$ whenever all its countable subgroups are $\mathfrak X$-groups.
\end{definition}

\begin{lemma}\label{lem:perfect-action}
Let $M$ be a normal abelian subgroup of the group $G$, and let $Q\leq G$ normalize~$M$.  If $Q$ is perfect, then $[M,Q,Q]=[M,Q]$.
\end{lemma}
\begin{proof}
This is obvious from the Three Subgroup Lemma.
\end{proof}

\begin{definition}
Let $H\leq G$.  We write $H\sn_dG$ if there is a chain
\[
 H=H_d\trianglelefteq H_{d-1}\trianglelefteq\ldots\trianglelefteq H_0=G.
\]
The least possible $d$, when it exists, is the ({\it subnormal}) {\it defect} of $H$ in $G$ and is denoted by~$\defect_G(H)$.
\end{definition}

We now prove a simple lemma about normal closures of subgroups. Recall that if $K\leq G$, the \emph{iterated normal closures} of $K$ in $G$ is defined recursively by
$$
K^{G,0}=G,\qquad K^{G,1}=K^G,\qquad K^{G,n+1}=K^{\,K^{G,n}}
$$
for $n\geq1$.

\begin{lemma}\label{lem:ambient}
Let $K\leq A\leq B$ be groups, and let $n\geq1$. If $K^{A,n}$ is perfect, then $K^{B,n}$ is perfect.
\end{lemma}

\begin{proof}
Put $P_j=K^{A,j}$ and $Q_j=K^{B,j}$ for $0\leq j\leq n$. We first show that $P_j\leq Q_j$ for every $j$. For $j=0$ this is just $A\leq B$. Suppose that $P_j\leq Q_j$ for some $j<n$. Then
$$
P_{j+1}=K^{P_j}\leq K^{Q_j}=Q_{j+1},
$$
and hence $P_j\leq Q_j$ for all $0\leq j\leq n$.

We next claim that $Q_j=P_j^{Q_{j-1}}$ for every $1\leq j\leq n$. Indeed, since $P_j\leq Q_j$ and $Q_j\trianglelefteq Q_{j-1}$, we have $P_j^{Q_{j-1}}\leq Q_j$. On the other hand, $P_j$ contains $K$, so $P_j^{Q_{j-1}}$ is a normal subgroup of $Q_{j-1}$ containing $K$, so that $Q_j\leq P_j^{Q_{j-1}}$. We proved that $Q_n=P_n^{Q_{n-1}}$. However, since $P_n$ is perfect, $P_n^{Q_{n-1}}\le Q_n'$. Hence $Q_n=Q_n'$, so $K^{B,n}$ is perfect.
\end{proof}

\begin{definition}
A \emph{general abelian series} of a group $G$ is a linearly ordered family $\mathcal S$ of subgroups which contains $\{1\}$ and $G$, is closed under unions and intersections of subchains, and has the following properties.
\begin{enumerate}[label=(\roman*)]
\item If $A<B$ are consecutive terms of $\mathcal S$, then $A\trianglelefteq B$ and $B/A$ is abelian.
\item Every non-identity element of $G$ lies in some $B\setminus A$, for consecutive terms $A$ and $B$ of the series.
\end{enumerate}
In this case, we call the couple $(A,B)$ a {\it jump} and say that an element $x$ of $G$ is {\it covered by} $(A,B)$ if $x\in B\setminus A$. The group $G$ belongs to $\SJ$ if it has a general abelian series every term of which is subnormal of finite defect in $G$.
\end{definition}

It is easy to see that every subgroup of an $\SJ$-group is $\SJ$, and that the class of $\SJ$-groups is closed with respect to forming extensions. However, the class $\SJ$ is not closed with respect to forming quotients, as one can see taking into account a free group of rank $2$. The following easy property is also remarkable.

\begin{lemma}\label{lem:freeproducts}
A free product of $\SJ$-groups is an $\SJ$-group.
\end{lemma}
\begin{proof}
Let $G=*_{i\in I}G_i$. By the universal property of the free product, the canonical embeddings of the factors $G_i$ into the restricted direct product
$$
D=\bigoplus_{i\in I}G_i
$$
induce a natural epimorphism $G\to D$.  By the Kurosh subgroup theorem the kernel $K$ is free, and hence belongs to $\SJ$ (because free groups are residually nilpotent). Thus, $G\in\SJ$ because $G/K$ is obviously $\SJ$.
\end{proof}

\begin{lemma}\label{lem:pair-obstruction}
Let $K=\langle x,y\rangle\leq G$ and suppose that $K^{G,n}$ is perfect.  Let
\[
 A_x<B_x,\qquad A_y<B_y
\]
be the jumps of one general abelian series of $G$ which cover $x$ and $y$, respectively.  Then $B_x$ and $B_y$ cannot both have defect at most $n$ in $G$.
\end{lemma}
\begin{proof}
Assume the contrary and let $(A,B)$ be the higher of the two jumps. Then $K\leq B$, so also \( K^{G,n}\leq B. \) Since $K^{G,n}$ is perfect, so \( K^{G,n}\leq B'\leq A. \) But the element covered by~\hbox{$(A,B)$} belongs to $K\leq K^{G,n}$, contradicting the fact that it lies outside $A$.
\end{proof}

We shall use an auxiliary hierarchy which is convenient because it is stable under the subgroup and quotient operations occurring later.  Put $\Hclass_0$ equal to the class of abelian groups.  Inductively, let $\Hclass_{h+1}$ consist of the groups $G$ which possess a complete linearly ordered system of normal subgroups whose factors belong to $\Hclass_h$.  We say that $G$ has \emph{finite normal height} if $G\in\Hclass_h$ for some $h<\omega$.

In the following theorem, the first two properties are elementary. The third follows from the classical constructions of Kov\'acs--Neumann and Hall \cite{Hall1962,KovacsNeumann1965}. We recall the relevant construction, since it also yields the additional property~(iv) needed below.

\begin{theorem}\label{thm:finite-height}
The hierarchy $\{\Hclass_h:h<\omega\}$ has the following properties:
\begin{enumerate}[label=\textnormal{(\roman*)}]
\item subgroups and factor groups of groups of finite normal height again have finite normal height;

\item restricted direct products of groups whose heights have a common finite bound have finite normal height;

\item for every positive integer $h$ there exists a non-trivial perfect group of finite normal height at least $h$;

\item if $H$ is a non-trivial perfect group of finite normal height and $1\neq a\in H$, there is a perfect group $T(H)$ of finite normal height, containing $H$, such that $\langle a\rangle^{T(H)}$ is perfect and
\[
 \langle a\rangle^{H,j}\text{ is perfect}
 \quad\Longrightarrow\quad
 \langle a\rangle^{T(H),j+1}\text{ is perfect}
\]
for every $j\geq1$.
\end{enumerate}
\end{theorem}
\begin{proof}
Parts~(i) and~(ii) follow easily by induction on the height. For~(iii) and~(iv), we recall the construction of Kov\'acs--Neumann \cite{KovacsNeumann1965}, together with a result of Hall \cite{Hall1962}. Let $F$ be a field of characteristic zero, and let $M=M(\mathbb Q,F)$ be the corresponding McLain group, in the notation of \cite[p.~14]{RobinsonFGSG2}. The group $M$ is non-trivial, perfect, torsion-free and locally nilpotent. Since every locally nilpotent group has a central series, $M$ has normal height~$1$. Moreover, every torsion-free locally nilpotent group is right-orderable, so we may fix a right-invariant total order on $M$. Let $\Lambda$ denote the set $M$ endowed with this order. Then the action of $M$ on $\Lambda$ by right multiplication is transitive and order-preserving. Now, let $H$ be a perfect group of normal height $h\geq1$, and form the wreath power $W=\operatorname{Wr}_{\Lambda} H$. Clearly, since $H$ is perfect, so is~$W$. Moreover, $W$ has a normal series whose factors are direct products of copies of $H$, so every such factor has normal height at most $h$, and hence $W\in\Hclass_{h+1}$. Thus $W$ has normal height either $h$ or $h+1$.

If $W$ has height $h+1$, the successor step is complete.  Suppose that $W$ has height $h$.  Let~$M$ act on $\Lambda$ by right multiplication. Since this action is transitive and order-preserving, it induces an action of $M$ on $W$.  Form the semidirect product $T(H)=W\rtimes M$. Theorem~D of \cite{Hall1962} implies that $W$ is a minimal normal subgroup of $T(H)$. Since both $W$ and $M$ are perfect, $T$ is perfect.  Also $T$ has finite normal height. It remains to determine its height.  Since $W$ is a minimal normal subgroup of $T$, every complete normal system of $T$ has a factor containing an isomorphic copy of $W$.  Such a factor cannot have normal height less than $h$, by part~(i).  Hence $T$ cannot have normal height less than $h+1$.  On the other hand, the normal series $\{1\}<W<T$ shows that its height is at most $h+1$.  Therefore $T$ has normal height $h+1$.  Starting with $M$ and repeating this construction gives non-trivial perfect groups of every finite normal height, and proves~(iii).

We now prove~(iv). Now
\[
1\neq a\in H\leq W\trianglelefteq T(H).
\]
Hence $\langle a\rangle^{T(H)}$ is a non-trivial normal subgroup of $T(H)$ contained in $W$.  By the minimality of $W$, $\langle a\rangle^{T(H)}=W$, and therefore $\langle a\rangle^{T(H)}$ is perfect.

Suppose finally that $\langle a\rangle^{H,j}$ is perfect for some $j\geq1$.  Since $H\leq W$, \cref{lem:ambient} implies that $\langle a\rangle^{W,j}$ is perfect.  Since $\langle a\rangle^{T(H)}=W$, we have $\langle a\rangle^{T(H),2}=\langle a\rangle^W$. Applying the recursive definition repeatedly then gives $\langle a\rangle^{T(H),j+1} =\langle a\rangle^{W,j}$ for every $j\geq1$. Thus $\langle a\rangle^{T(H),j+1}$ is perfect, and~(iv) follows.
\end{proof}

\begin{lemma}\label{lem:heightSJ}
Every group of finite normal height belongs to $\SJ$.
\end{lemma}

\begin{proof}
Induct on the height.  The assertion is clear for $\Hclass_0$.  Let $G\in\Hclass_{h+1}$ and choose a complete normal system in $G$ whose factors lie in $\Hclass_h$.  By induction each factor is an $\SJ$-group.  Refine every jump by lifting an $\SJ$-series of the corresponding factor.  A term of defect $d$ in a factor lifts to a subgroup of defect at most $d+1$ in $G$.  The resulting refinement is a general abelian series all of whose terms have finite defect in $G$.
\end{proof}

\section{The factors of the construction}

In this section we introduce the groups that will serve as the factors in the construction of the final group.  The aim is to obtain, for each prescribed depth, a group of finite normal height containing two distinguished involutions whose joint iterated normal closures remain perfect for that many steps, while the normal closure of either involution separately has a controlled non-trivial abelian quotient.  These two features will later allow us to force the required behaviour under amalgamation and, at the same time, to distinguish the normal closures of the individual involutions from those generated by the pair.

\smallskip

For a group $H$, let
\[
 E(H)=\bigoplus_{h\in H}\langle e_h\rangle,
 \qquad |e_h|=2,
\]
written additively, and let $H$ act by the right regular permutation $e_x^h=e_{xh}$. Put
\[
 I(H)=\langle e_1+e_h:h\in H\rangle\leq E(H).
\]
Thus $I(H)$ is an elementary abelian normal subgroup of $E(H)\rtimes H$.

\begin{lemma}\label{lem:regular-permutation}
If $H$ is perfect, then $[I(H),H]=I(H)$.
\end{lemma}

\begin{proof}
Let $\overline I=I(H)/[I(H),H]$.  For $h\in H$ write
\[
 \xi_h=(e_1+e_h)+[I(H),H]\in\overline I.
\]
Since conjugate elements of $I(H)$ have the same image in $\overline I$, for $h,k\in H$ we have
\[
\begin{aligned}
\xi_{hk}
&=(e_1+e_{hk})+[I(H),H]\\
&=(e_1+e_k)+(e_k+e_{hk})+[I(H),H]\\
&=\xi_k+\xi_h
=\xi_h+\xi_k.
\end{aligned}
\]
Thus $h\mapsto\xi_h$ is a homomorphism from $H$ to the abelian group $\overline I$.  As $H=H'$, this homomorphism is trivial.  The elements $e_1+e_h$ generate $I(H)$, so $\overline I=\{0\}$ and the statement is proved.
\end{proof}

\begin{lemma}\label{lem:seed}
There is a perfect group of finite normal height containing a non-trivial involution.
\end{lemma}

\begin{proof}
Choose a non-trivial perfect group $H$ of finite normal height from \cref{thm:finite-height}.  By \cref{lem:regular-permutation}, $[I(H),H]=I(H)$. Therefore $P_0=I(H)\rtimes H$ is perfect.  It has finite normal height because $I(H)$ is abelian and normal and $P_0/I(H)\simeq H$.  Every non-zero element of~$I(H)$ is an involution.
\end{proof}

\begin{corollary}\label{cor:deep-involutions}
For every $N\geq1$ there exist a perfect group~$P_N$ of finite normal height and an involution $a_N\in P_N$ such that $\langle a_N\rangle^{P_N,j}$ is perfect for $1\leq j\leq N$.
\end{corollary}
\begin{proof}
Let $P_0$ be the perfect group supplied by \cref{lem:seed}, and let $a\in P_0$ be a non-trivial involution.  By part~(iv) of \cref{thm:finite-height}, there is a perfect overgroup $P_1=T(P_0)$ of finite normal height such that $\langle a\rangle^{P_1}$ is perfect.

Suppose inductively that, for some $m\geq1$, we have constructed a perfect group $P_m$ of finite normal height containing $a$ and such that $\langle a\rangle^{P_m,j}$ is perfect for every $1\leq j\leq m$.  Apply part~(iv) of \cref{thm:finite-height} to $P_m$ and put $P_{m+1}=T(P_m)$. Since $P_m\leq P_{m+1}$, \cref{lem:ambient} shows that $\langle a\rangle^{P_{m+1},j}$ is perfect for $1\leq j\leq m$.  Moreover, applying part~(iv) with $j=m$ gives that $\langle a\rangle^{P_{m+1},m+1}$ is perfect.  Thus the assertion holds for $m+1$.

Iterating this construction $N$ times and putting $a_N=a$ gives the required group $P_N$ and involution $a_N$.
\end{proof}

Fix $N\geq1$.  Take two copies $(P,a)$ and $(R,b)$ supplied by \cref{cor:deep-involutions}.  Put
\[
 Q_j=\langle a\rangle^{P,j},\qquad
 S_j=\langle b\rangle^{R,j}
 \qquad(1\leq j\leq N).
\]
Every $Q_j$ and every $S_j$ is perfect. Let
\[
 U=\bigoplus_{p\in P}\langle u_p\rangle,
 \qquad
 V=\bigoplus_{r\in R}\langle v_r\rangle
\]
be elementary abelian $2$-groups.  Let $P$ act regularly on $U$ and trivially on $V$, and let $R$ act regularly on $V$ and trivially on $U$.  Set
\[
 A_N=(U\oplus V)\rtimes(P\times R).
\]
This group has finite normal height, and hence belongs to $\SJ$. Define
\[
 u=u_1+u_a\in U,\qquad w=v_1+v_b\in V,
\]
and, in the semidirect product, $x=(w,a)$ and $y=(u,b)$. In particular, $\langle x,y\rangle$ is a Klein $4$-group. For $1\leq j\leq N$ define
\[
 U_j=[U,Q_1,\ldots,Q_j],\qquad
 V_j=[V,S_1,\ldots,S_j].
\]

\begin{proposition}\label{prop:pairhulls}
For $1\leq j\leq N$, \( \langle x,y\rangle^{A_N,j} = (U_j\oplus V_j)\rtimes(Q_j\times S_j). \) In particular, all these $N$ iterated normal closures are perfect.
\end{proposition}

\begin{proof}
For $j=1$, the normal closure in $P\times R$ of the projections $a,b$ is $Q_1\times S_1$, while
\[
 [U\oplus V,Q_1\times S_1]=[U,Q_1]\oplus[V,S_1]=U_1\oplus V_1.
\]
Moreover
\[
 u=u_1+u_a=[u_1,a]\in[U,Q_1],
 \qquad
 w=v_1+v_b=[v_1,b]\in[V,S_1].
\]
This gives the formula for $j=1$.

Assume the formula holds at level $j-1$.  Inside the new ambient group
\[
 (U_{j-1}\oplus V_{j-1})\rtimes(Q_{j-1}\times S_{j-1}),
\]
the normal closure of $a,b$ in the quotient part is $Q_j\times S_j$, and the commutator subgroup contributed by the abelian normal part is
\[
 [U_{j-1},Q_j]\oplus[V_{j-1},S_j]=U_j\oplus V_j.
\]
We must only check that the tags $u,w$ still lie in these smaller groups.  Since $a\in Q_j$ and~$Q_j$ is perfect,
\[
 u=[u_1,a]\in[U,Q_j]
 =[U,\underbrace{Q_j,\ldots,Q_j}_{j\text{ times}}]
 \leq[U,Q_1,\ldots,Q_j]=U_j
\]
by \cref{lem:perfect-action} and the inclusions $Q_j\leq Q_i$ for $i\leq j$.  Similarly $w\in V_j$, and the result is proved.

Finally, since $Q_j$ and $S_j$ are perfect and \cref{lem:perfect-action} gives
\[
 [U_j,Q_j]=U_j,
 \qquad
 [V_j,S_j]=V_j,
\]
the derived subgroup of $(U_j\oplus V_j)\rtimes(Q_j\times S_j)$ contains both the abelian normal part and the quotient part.  Hence the whole group is perfect.
\end{proof}

\begin{lemma}\label{lem:cheapmarker}
Let $X=x^{A_N}$.  Put
\[
 C=[U,Q_1],
 \qquad
 W_b=\langle v_r+v_{br}:r\in R\rangle\leq V.
\]
Then
\[
 X=(C\oplus W_b)\rtimes Q_1,
 \qquad
 X'=C\rtimes Q_1,
\]
and therefore
\[
 X/X'\simeq W_b.
\]
In particular $x\notin X'$.  The analogous statement holds for $Y=y^{A_N}$.
\end{lemma}

\begin{proof}
The projection of $X$ onto $P\times R$ is $Q_1\times1$.  Commutators with $U$ give $C\leq X$.  Modulo $C$, the conjugates of $x$ have the form
\[
 (w^r,a^p)\qquad(p\in P,\ r\in R).
\]
The subgroup $Q_1$ acts trivially on $W_b$, and $Q_1$ is perfect.  It follows that the derived subgroup of $X/C$ projects onto $Q_1$ and has trivial $W_b$-coordinate.  Hence $1\times Q_1\leq X/C$.  Removing these pure $Q_1$-components from the displayed conjugates leaves all $w^r$, which generate $W_b$.  Thus \( X=(C\oplus W_b)\rtimes Q_1. \) Now $[C,Q_1]=C$ by \cref{lem:perfect-action}, while $Q_1$ acts trivially on $W_b$.  Therefore $X'=C\rtimes Q_1$, as required.
\end{proof}

\begin{lemma}\label{lem:independent-involutions}
In $X/X'\simeq W_b$, the images of the conjugates of $x$ coming from the $R$-action are
\[
 v_r+v_{br}\qquad(r\in R).
\]
Two such elements are equal if and only if $r$ and $s$ lie in the same left coset of $\langle b\rangle$.  Distinct such elements are linearly independent in the elementary abelian $2$-group $W_b$.  The symmetric statement holds for $y$.
\end{lemma}

\begin{proof}
The equality for $r$ and $br$ follows from $b^2=1$.  If $\langle b\rangle r\neq\langle b\rangle s$, then the two supports
\[
 \{r,br\},\qquad \{s,bs\}
\]
are disjoint.  Hence the corresponding order-two elements are independent.  The same support argument applies to any finite family of distinct cosets.
\end{proof}

\section{Iterated amalgamated products}\label{sec:amalgams}

In this section we record the facts about iterated amalgamated products that are needed in the construction of our counterexample.  We allow the family of factors to be arbitrary, and in particular uncountable.  More precisely, let $(F_\alpha)_{\alpha<\kappa}$ be a non-empty family of groups indexed by an ordinal $\kappa$.  Put $P_0=\{1\}$ and $P_1=F_0$.  For every ordinal $\alpha$ with $0<\alpha<\kappa$, once $P_\alpha$ has been constructed, adjoin $F_\alpha$ by forming a free product with amalgamation $P_{\alpha+1}=P_\alpha*_{C_\alpha}F_\alpha$, where the copy of $C_\alpha$ in $P_\alpha$ is contained in one of the factors $F_\beta$ with $\beta<\alpha$.  If $\lambda\leq\kappa$ is a non-zero limit ordinal, put $P_\lambda=\bigcup_{\alpha<\lambda}P_\alpha$, using the natural embeddings supplied at the preceding stages.  Finally put $G=P_\kappa$.  We call any group obtained in this way an \emph{iterated amalgam}.  The factors and the successive amalgamations form a tree, so this is precisely a tree product in the sense used by Fischer \cite{Fischer1975}.  In particular, each new factor is joined to the part already constructed by a single amalgamation; the same subgroup may, of course, occur in more than two factors after several identifications have been performed.

\begin{lemma}\label{lem:subgroup-amalgam}
Let $G$ be an iterated amalgam with factor groups $F_i$ and amalgamated subgroups $C_e$, and let $H\leq G$.  Put
\[
N=
\left\langle
H\cap gF_i g^{-1}:i,\ g\in G
\right\rangle^H.
\]
Then $N\trianglelefteq H$ and $H/N$ is free.  Moreover, $N$ is an iterated amalgam whose factors are of the form $H\cap gF_i g^{-1}$. If $D$ is an amalgamated subgroup in this decomposition contained in a factor $H\cap gF_i g^{-1}$, then there exist an amalgamated subgroup $C_e\leq F_i$ of $G$ and an element $f\in F_i$ such that $D\leq H\cap gfC_ef^{-1}g^{-1}$.
\end{lemma}

\begin{lemma}\label{lem:subgroup-amalgam}
Let $G$ be an iterated amalgam with factor groups $F_i$ and amalgamated subgroups $C_e$, and let $H\leq G$.  Put
\[
N=
\left\langle
H\cap gF_i g^{-1}:i,\ g\in G
\right\rangle^H.
\]
Then $N\trianglelefteq H$ and $H/N$ is free.  Moreover, $N$ is an iterated amalgam whose factors are of the form $H\cap gF_i g^{-1}$. If $D$ is an amalgamated subgroup in this decomposition contained in a factor $H\cap gF_i g^{-1}$, then there exist an amalgamated subgroup $C_e\leq F_i$ of $G$ and an element $f\in F_i$ such that $D\leq H\cap gfC_ef^{-1}g^{-1}$.
\end{lemma}
\begin{proof}
Let $I$ be the index set of the factors $F_i$ of $G$. The normality of $N$ in $H$ is immediate from its definition.

We now use Fischer's subgroup theorem \cite[Theorem~5]{Fischer1975} and \cite[Corollary~5.1]{Fischer1975}.  Applied to $H\leq G$, \cite[Theorem~5]{Fischer1975} gives an iterated amalgam $S$.  For each $i\in I$, choose a set $\mathcal R_i$ containing exactly one representative $r$ from each double coset $HgF_i$, $g\in G$.  The factors of $S$ are
\[
H\cap rF_i r^{-1},
\qquad i\in I,\quad r\in\mathcal R_i.
\]

By \cite[Corollary~5.1]{Fischer1975}, $H/S^H$ is free.  We claim that $S^H=N$.  Every factor of $S$ is one of the subgroups used in the definition of $N$, and hence $S^H\leq N$.  Conversely, let $i\in I$ and $g\in G$.  There is a unique $r\in\mathcal R_i$ such that $HgF_i=HrF_i$.  Thus $g=hrf$ for some $h\in H$ and $f\in F_i$, and
\[
H\cap gF_i g^{-1}
=
h\bigl(H\cap rF_i r^{-1}\bigr)h^{-1}
\leq S^H.
\]
Hence every subgroup used in the definition of $N$ is contained in $S^H$.  Therefore $N=S^H$, and $H/N$ is free.

We now apply \cite[Theorem~5]{Fischer1975} and \cite[Corollary~5.1]{Fischer1975} to $N\leq G$.  The family $H\cap gF_i g^{-1}$, with $i\in I$ and $g\in G$, is invariant under conjugation by $H$, since conjugation by $h\in H$ sends $H\cap gF_i g^{-1}$ to $H\cap hgF_i(hg)^{-1}$.  Since $N$ is the normal closure in $H$ of this family, $N$ is generated by the subgroups $H\cap gF_i g^{-1}$.  Each of these subgroups is contained in $N$, and, since $N\leq H$,
\[
N\cap gF_i g^{-1}
=
H\cap gF_i g^{-1}
\qquad (i\in I,\ g\in G).
\]
Thus $N$ is generated by its intersections with the conjugates of the factors of $G$.

Applied to $N\leq G$, \cite[Theorem~5]{Fischer1975} gives an iterated amalgam $T$.  For each $i\in I$, choose a set $\mathcal S_i$ containing exactly one representative $s$ from each double coset $NgF_i$, $g\in G$.  The factors of $T$ are
\[
N\cap sF_i s^{-1}
=
H\cap sF_i s^{-1},
\qquad i\in I,\quad s\in\mathcal S_i.
\]
Since $N$ is generated by its intersections with the conjugates of the factors of $G$, \cite[Corollary~5.1]{Fischer1975} gives $N=T$. Thus $N$ is an iterated amalgam with factors of the required form.

It remains to describe its amalgamated subgroups.  Let $D$ be an amalgamated subgroup contained in the factor $N\cap sF_i s^{-1}$, where $i\in I$ and $s\in\mathcal S_i$.  By \cite[Theorem~5(1)]{Fischer1975}, there are $j\in I$, a representative $t\in\mathcal S_j$, and an amalgamated subgroup $C_e\leq F_i$ joining $F_i$ and $F_j$, such that $D$ is the subgroup amalgamated between $N\cap sF_i s^{-1}$ and $N\cap tF_jt^{-1}$.  In Fischer's notation,
\[
s=t,\qquad
s=tE_{ji},\qquad\text{or}\qquad
t=sE_{ij}.
\]
Here $E_{ij}\in F_i$ and $E_{ji}\in F_j$.

By \cite[Theorem~5(2)]{Fischer1975}, in the first two cases the amalgamated subgroup $D$ is $N\cap sC_es^{-1}$.  In the third case it is $N\cap tC_et^{-1}$. Since $t=sE_{ij}$ and $E_{ij}\in F_i$, this has the required form.  Thus, taking $f=1$ in the first two cases and $f=E_{ij}$ in the third,
\[
D
=
N\cap sfC_ef^{-1}s^{-1}
\leq
H\cap sfC_ef^{-1}s^{-1},
\]
as required.
\end{proof}

\begin{corollary}\label{cor:free-subgroup-tree-product}
Let $H\leq G$, where $G$ is an iterated amalgam with factor groups $F_i$. If $H\cap gF_i g^{-1}=\{1\}$ for every $i$ and every $g\in G$, then $H$ is free.
\end{corollary}

\begin{lemma}\label{lem:factor-quotient}
Let $G$ be an iterated amalgam with factor groups $F_i$.  For each $i$ let $N_i\trianglelefteq F_i$, and suppose that $N_i\cap C=\{1\}$ for every amalgamated subgroup $C$ contained in $F_i$.  Then the quotient maps $F_i\longrightarrow F_i/N_i$ induce an epimorphism $\varphi:G\longrightarrow\overline G$, where $\overline G$ is the corresponding iterated amalgam of the groups $F_i/N_i$.  Moreover, $\ker\varphi$ is a free product of conjugates of the groups $N_i$ and a possibly trivial free group.
\end{lemma}
\begin{proof}
The condition $N_i\cap C=1$ shows that every amalgamated subgroup still embeds after passing from $F_i$ to $F_i/N_i$.  Hence the quotient maps on the factors define the required epimorphism. Put $D=\ker\varphi$.  Then $D\cap F_i=N_i$, and, since $D\trianglelefteq G$,
\[
 D\cap gF_i g^{-1}=gN_i g^{-1}
\]
for every $g\in G$. Apply \cref{lem:subgroup-amalgam} to the subgroup $D\leq G$. Since $D$ is the normal closure in~$G$ of the subgroups $N_i$, it is generated by the groups $D\cap gF_i g^{-1}=gN_i g^{-1}$.  Hence the normal subgroup appearing in the statement of \cref{lem:subgroup-amalgam} is all of $D$.  Therefore we can express~$D$ as an iterated amalgam of conjugates of the $N_i$.  Since $D$ meets every conjugate of every amalgamated subgroup trivially, all the subgroups used in these amalgamations are trivial.  Thus $D$ is a free product of conjugates of the groups $N_i$.
\end{proof}

\begin{lemma}\label{lem:abelian-compression}
Let $G$ be an iterated amalgam with factor groups $F_i$, $i\in I$, where every $F_i$ is an elementary abelian $2$-group and every amalgamated subgroup is either trivial or of order~$2$. For each $i$, let $Y_i$ be the set of non-identity elements of the non-trivial amalgamated subgroups contained in $F_i$, and assume that $Y_i$ is linearly independent over $\mathbb F_2$.

Let
\[
V=\bigoplus_{i\in I}F_i,
\]
and let $R$ be the subspace of $V$ generated by all elements $c+d$ such that $\langle c\rangle\leq F_i$ and $\langle d\rangle\leq F_j$ are identified in forming $G$.  Put $L=V/R$.  Then\textnormal:
\begin{enumerate}[label=\textnormal{(\roman*)}]
\item the natural map $F_i\to L$ is injective for every $i$\textnormal;

\item the maps $F_i\to L$ induce an epimorphism $G\to L$ whose kernel is free\textnormal;

\item suppose, in addition, that for each $i$ a family $\mathcal A_i$ of subgroups of order $2$ of $F_i$ is specified. Let $X_i$ be the set consisting of $Y_i$ together with the non-identity elements of the members of $\mathcal A_i$, and assume that $X_i$ is linearly independent over $\mathbb F_2$. If $A\in\mathcal A_i$ and $B\in\mathcal A_j$, then the images of $A$ and $B$ in $L$ are equal if and only if their non-identity elements are connected by a chain of the prescribed identifications.  Moreover, if one subgroup is chosen from each resulting equivalence class in $\bigcup_{i\in I}\mathcal A_i$, then the non-identity elements of the chosen subgroups have linearly independent images in $L$.
\end{enumerate}
\end{lemma}

\begin{proof}
For each $i$, choose a basis $B_i$ of $F_i$ containing $Y_i$, and, under the hypotheses of~\textnormal{(iii)}, containing $X_i$.  This is possible by the assumed linear independence.  Let $B$ be the disjoint union of the sets $B_i$.

Define an equivalence relation $\sim$ on $B$ by declaring $c\sim d$ whenever $\langle c\rangle$ and $\langle d\rangle$ are identified in one of the free products with amalgamation used to construct $G$, and then taking the equivalence relation generated by these identifications. Since the construction of $G$ contains no cycles, no equivalence class contains two distinct elements of the same~$B_i$.   The relations defining $R$ identify exactly the basis elements which belong to the same $\sim$-class.  Hence $L$ has a basis consisting of one element for each $\sim$-class in $B$.  Since no class contains two distinct elements of the same $B_i$, the restriction $F_i\to L$ is injective for every $i$.  This proves~\textnormal{(i)}.

The maps $F_i\to L$ agree on every pair of subgroups identified in forming $G$.  They therefore induce a homomorphism $\pi:G\to L$.  Since $L$ is generated by the images of the factors, $\pi$ is surjective.  Put $D=\ker\pi$.  By~\textnormal{(i)}, $D\cap F_i=\{1\}$ for every $i$.  Since $D\trianglelefteq G$, we also have $D\cap gF_i g^{-1}=\{1\}$ for every $i$ and every $g\in G$.  Hence $D$ is free by \cref{cor:free-subgroup-tree-product}.  This proves \textnormal{(ii)}.

Finally, assume the hypotheses of~\textnormal{(iii)}, and let $q:V\to L$ be the quotient map.  If $A=\langle a\rangle\in\mathcal A_i$, then $a\in X_i\subseteq B_i$. By~\textnormal{(i)}, $q(a)\neq 0$, and the image of $A$ in $L$ is the subgroup $\langle q(a)\rangle$ of order $2$.

Let $A=\langle a\rangle\in\mathcal A_i$ and $B=\langle b\rangle\in\mathcal A_j$.  Since $L$ is an elementary abelian $2$-group, the images of $A$ and $B$ in $L$ are equal if and only if $q(a)=q(b)$.  The quotient $L=V/R$ has a basis indexed by the $\sim$-classes of the elements of $B$.  Hence $q(a)=q(b)$ if and only if $a$ and $b$ belong to the same $\sim$-class.  By the definition of $\sim$, this is equivalent to saying that $a$ and $b$ are connected by a chain of the prescribed identifications.

Now choose one subgroup from each equivalence class in $\bigcup_{i\in I}\mathcal A_i$, and let $a_1,\ldots,a_n$ be the non-identity elements of any finite collection of the chosen subgroups.  The elements $a_1,\ldots,a_n$ belong to distinct $\sim$-classes, so $q(a_1),\ldots,q(a_n)$ are distinct elements of the basis of $L$ indexed by those classes.  They are therefore linearly independent.  Since every linear dependence involves only finitely many elements, the images of the non-identity elements of all the chosen subgroups are linearly independent.  This proves \textnormal{(iii)}.
\end{proof}

\begin{definition}
Let $h\geq 0$.  An \emph{admissible amalgam of height at most $h$} is an iterated amalgam whose factor groups are either groups in $\Hclass_h$ or elementary abelian $2$-groups, whose amalgamated subgroups are trivial or of order $2$, and satisfying the following property: if~$E$ is an elementary abelian factor, then we require the set $Y_E$ of non-identity elements of the distinct non-trivial amalgamated subgroups contained in $E$ to be linearly independent over $\mathbb F_2$.
\end{definition}

\begin{theorem}\label{thm:amalgam}
Every admissible amalgam of finite height belongs to $\SJ$.
\end{theorem}
\begin{proof}
We argue by induction on $h$. By \cref{lem:freeproducts}, we may assume that every amalgamated subgroup has order $2$.

\smallskip
\noindent\emph{The case $h=0$.} Let
\[
V=\bigoplus_F F,
\]
where $F$ ranges over the factors of the amalgamated free product $G$, and let $R$ be the subgroup of $V$ generated by all elements $c-c'$ for which $\langle c\rangle\leq F$ and $\langle c'\rangle\leq F'$ are the two copies of the same amalgamated subgroup.  Put $L=V/R$, and note that every factor~$F$ embeds in $L$.  The natural maps of the factors into $L$ agree on the prescribed amalgamated subgroups and therefore induce an epimorphism $\pi:G\to L$.  Let $D=\ker\pi$.  Since $\pi$ is injective on every factor, $D\cap F=\{1\}$ for every factor $F$.  As $D\trianglelefteq G$, we also have $D\cap gFg^{-1}=\{1\}$ for every $g\in G$.  Hence $D$ is free by \cref{cor:free-subgroup-tree-product}.  Thus $D\in\SJ$ and $L\in\SJ$, and so $G\in\SJ$.

\smallskip
\noindent\emph{The induction step.} Assume that $h>0$ and that the theorem holds for admissible amalgams of height at most $h-1$.

For every factor $F\in\Hclass_h$, choose a complete linearly ordered system $\mathcal S_F$ of normal subgroups of $F$ whose factors belong to $\Hclass_{h-1}$.  Let $\Lambda_F$ be the set of jumps of $\mathcal S_F$, with the order induced by $\mathcal S_F$.  Thus, if $(A,B)$ and $(A',B')$ are distinct jumps, we put
\[
(A,B)<(A',B')
\quad\text{if and only if}\quad
B\leq A'.
\]

If $C=\langle c\rangle$ is a non-trivial amalgamated subgroup contained in $F$, completeness of~$\mathcal S_F$ determines a unique jump $(A,B)$ such that $c\in B\setminus A$. We denote this jump by $\lambda_F(C)\in\Lambda_F$.

If $E$ is an elementary abelian factor, let $Y_E$ be the set of non-identity elements of its distinct non-trivial amalgamated subgroups.  By hypothesis, $Y_E$ is linearly independent over $\mathbb F_2$, and hence it is contained in a basis $B_E$ of $E$. Choose an arbitrary linear order on $B_E$ and put $\Lambda_E=B_E$ with this order.  If $C=\langle c\rangle\leq E$ is a non-trivial amalgamated subgroup, then $c\in Y_E\subseteq B_E$, and we put $\lambda_E(C)=c$.

We now combine these linearly ordered sets into a single linearly ordered set.  Let
\[
\Omega=\bigsqcup_F\Lambda_F
\]
be the disjoint union of the sets $\Lambda_F$, where $F$ ranges over the factor groups of $G$. Define a symmetric relation $\sigma$ on $\Omega$ as follows.  For $x,y\in\Omega$, write $x\,\sigma\,y$ if there is a non-trivial amalgamated subgroup $C$ joining two factors $F$ and $F'$ such that
\[
\{x,y\}
=
\{\lambda_F(C),\lambda_{F'}(C)\}.
\] Let $\sim$ be the equivalence relation on $\Omega$ generated by
$\sigma$, and put
\(
\Omega_0=\Omega/{\sim}.
\)

For each factor $F$, the natural map
\[
\Lambda_F\longrightarrow\Omega_0,
\qquad
x\longmapsto[x],
\]
is injective. We may therefore regard each $\Lambda_F$ as a subset of $\Omega_0$. Let $\preccurlyeq$ be the relation on $\Omega_0$ generated by the orders on the sets $\Lambda_F$; that is, $\preccurlyeq$ is the transitive closure of all relations
\[
x\preccurlyeq y
\qquad
(x,y\in\Lambda_F,\ x\leq_F y).
\]
This is a partial order.  For if antisymmetry failed, a shortest chain of inequalities witnessing both $x\preccurlyeq y$ and $y\preccurlyeq x$, with $x\neq y$, would pass through a non-trivial closed sequence of factors.  Removing immediate backtracking would produce a cycle among the factors, which is impossible. By the order-extension principle, choose a linear order on $\Omega_0$ extending $\preccurlyeq$, and denote the resulting linearly ordered set by $\Lambda$.  Thus the natural map $\iota_f:\Lambda_F\longrightarrow\Lambda$ is order preserving for every factor $F$, and whenever a non-trivial amalgamated subgroup $C$ joins two factors $F$ and $F'$, we have $\lambda_F(C)=\lambda_{F'}(C)$ as elements of $\Lambda$.

Let $I$ be an initial section of $\Lambda$.  We associate with every factor $F$ a normal subgroup $N_F(I)$. If $F\in\Hclass_h$, let $I_F=\iota_F^{-1}(I)\subseteq\Lambda_F$ and define
\[
N_F(I)=
\bigcup\{B:(A,B)\in I_F\},
\]
with the empty union interpreted as $\{1\}$.  Since $\mathcal S_F$ is complete, $N_F(I)$ is a term of $\mathcal S_F$. If $F=E$ is elementary abelian, define \( N_E(I)= \left\langle b\in B_E:\iota_E(b)\in I \right\rangle. \)

The construction of $\Lambda$ gives, for every non-trivial amalgamated subgroup $C$ contained in factors $F$ and $F'$,
\[
C\leq N_F(I)
\quad\Longleftrightarrow\quad
C\leq N_{F'}(I).
\tag{1}\label{eq:edge-compatible}
\]
Consequently, the quotient maps $F\to F/N_F(I)$ preserve all the prescribed subgroup identifications.  They therefore induce an epimorphism
\[
q_I:G\longrightarrow G_I,
\]
where $G_I$ is the iterated amalgam of the groups $F/N_F(I)$ with the images of the original amalgamated subgroups.  Put
\[
K_I:=\ker q_I=
\left\langle N_F(I):F\text{ a factor of }G\right\rangle^G.
\tag{2}\label{eq:KI-normal-closure}
\]

If $I\subseteq J$, then $N_F(I)\leq N_F(J)$ for every factor $F$, and hence $K_I\leq K_J$.  Moreover, $K_\varnothing=\{1\}$ and  $K_\Lambda=G$.

We next verify that the chain $\{K_I\}$ is complete.  Let $\{I_\nu\}$ be a non-empty chain of initial sections of $\Lambda$. For every factor $F$ we have
\[
N_F\left(\bigcup_\nu I_\nu\right)
 =\bigcup_\nu N_F(I_\nu),
\qquad
N_F\left(\bigcap_\nu I_\nu\right)
 =\bigcap_\nu N_F(I_\nu).
\tag{3}\label{eq:NF-continuity}
\]

Let $I=\bigcup_\nu I_\nu$.  By \eqref{eq:KI-normal-closure}, every element of $K_I$ is a product of finitely many conjugates of elements belonging to groups $N_F(I)$. By \eqref{eq:NF-continuity}, all these elements belong to $N_F(I_\nu)$ for one common $\nu$, since the $I_\nu$ form a chain. Hence
\[
K_{\bigcup_\nu I_\nu}
=
\bigcup_\nu K_{I_\nu}.
\tag{4}\label{eq:KI-union}
\]

Now let $I=\bigcap_\nu I_\nu$. Since $I\subseteq I_\nu$ for every $\nu$, we have $K_I\leq K_{I_\nu}$, and hence $K_I\leq\bigcap_\nu K_{I_\nu}$. For the reverse inclusion, suppose that $g\notin K_I$.  Then $q_I(g)\neq1$ in $G_I$, and choose a reduced form for $q_I(g)$.

We first record the following observation.  If $F$ is a factor and $C\leq F$ is an amalgamated subgroup, then
\[
\bigcap_\nu C N_F(I_\nu)=C N_F(I).
\]
Indeed, the inclusion from right to left is clear.  Conversely, suppose that $a\notin C N_F(I)$.  Then $c^{-1}a\notin N_F(I)=\bigcap_\nu N_F(I_\nu)$ for every $c\in C$. Since $C$ is finite, for each $c\in C$ we may choose $\nu(c)$ such that $c^{-1}a\notin N_F(I_{\nu(c)})$.  The groups $N_F(I_\nu)$ form a chain, so there is one index $\nu$ such that $N_F(I_\nu)\leq N_F(I_{\nu(c)})$ for every $c\in C$.  Hence $c^{-1}a\notin N_F(I_\nu)$ for every $c\in C$, and therefore $a\notin C N_F(I_\nu)$.

The reducedness of the chosen form of $q_I(g)$ is determined by finitely many conditions saying that certain elements do not belong to $N_F(I)$ or to $C N_F(I)$, for suitable factors $F$ and amalgamated subgroups $C$.  By the preceding observation, each of these conditions remains valid with $I$ replaced by some $I_\nu$. Since there are only finitely many conditions and the $I_\nu$ form a chain, one index $\nu$ may be chosen for which they all remain valid. The chosen word is therefore still reduced and non-trivial in $G_{I_\nu}$. Thus $g\notin K_{I_\nu}$, and hence
\[
K_{\bigcap_\nu I_\nu}
=
\bigcap_\nu K_{I_\nu}.
\tag{5}\label{eq:KI-intersection}
\]

It follows from \eqref{eq:KI-union} and \eqref{eq:KI-intersection} that the groups $K_I$, with $I$ ranging over the initial sections of $\Lambda$, form a complete linearly ordered system of normal subgroups of $G$.

Fix $\lambda\in\Lambda$, and put
\[
I_\lambda^-=\{\mu\in\Lambda:\mu<\lambda\},
\qquad
I_\lambda^+=\{\mu\in\Lambda:\mu\leq\lambda\}.
\]
Set $R_\lambda= K_{I_\lambda^+}/K_{I_\lambda^-}$. Then $R_\lambda$ is a normal subgroup of $G_{I_\lambda^-}=G/K_{I_\lambda^-}$.

Let $F$ be an original factor of $G$, and write $\overline F=F/N_F(I_\lambda^-)$ for its image in $G_{I_\lambda^-}$.  The natural map $G_{I_\lambda^-}\to G_{I_\lambda^+}$ restricts on $\overline F$ to the quotient map
\[
F/N_F(I_\lambda^-)
\longrightarrow
F/N_F(I_\lambda^+).
\]
Hence
\[
R_\lambda\cap\overline F
=
N_F(I_\lambda^+)/N_F(I_\lambda^-).
\tag{6}\label{eq:R-factor-intersection}
\]
If $\lambda$ does not belong to the image of $\Lambda_F$ in $\Lambda$, then the two initial sections $I_\lambda^-$ and $I_\lambda^+$ induce the same initial section of $\Lambda_F$.  Hence $N_F(I_\lambda^-)=N_F(I_\lambda^+)$, and the group above is trivial. Suppose now that $\lambda$ belongs to the image of $\Lambda_F$. Since $\Lambda_F\to\Lambda$ is injective, there is a unique $\mu\in\Lambda_F$ whose image is $\lambda$.  If $F\in\Hclass_h$, then $\mu$ is a jump $(A,B)$ of $\mathcal S_F$, and
\[
N_F(I_\lambda^-)=A,
\qquad
N_F(I_\lambda^+)=B.
\]
Thus $R_\lambda\cap\overline F\simeq B/A$, which belongs to $\Hclass_{h-1}$ by the choice of $\mathcal S_F$. If $F$ is elementary abelian, then $\mu$ is a basis element $b\in B_F$.  In this case
\[
N_F(I_\lambda^+)
=
\langle N_F(I_\lambda^-),b\rangle,
\]
and, since $b\notin N_F(I_\lambda^-)$, the quotient $N_F(I_\lambda^+)/N_F(I_\lambda^-)$ has order $2$. Hence every non-trivial group $R_\lambda\cap\overline F$ either belongs to $\Hclass_{h-1}$ or is cyclic of order $2$.

Let
\[
R_\lambda^0=
\left\langle
R_\lambda\cap g\overline Fg^{-1}:
\overline F\text{ a factor of }G_{I_\lambda^-},\
g\in G_{I_\lambda^-}
\right\rangle^{R_\lambda}.
\]
Since $R_\lambda\trianglelefteq G_{I_\lambda^-}$, $ R_\lambda\cap g\overline Fg^{-1} = g(R_\lambda\cap\overline F)g^{-1}. $ Thus all the groups used in the definition of $R_\lambda^0$ are conjugates of the groups described in~\eqref{eq:R-factor-intersection}. We now apply \cref{lem:subgroup-amalgam} to the subgroup $R_\lambda\leq G_{I_\lambda^-}$. Recall that the factor groups of~$G_{I_\lambda^-}$ are the groups $\overline F=F/N_F(I_\lambda^-)$, where $F$ ranges over the original factor groups of $G$; moreover, $R_\lambda$ is normal in $G_{I_\lambda^-}$. The lemma therefore gives that $R_\lambda/R_\lambda^0$ is free.

We next determine the factor groups occurring in the decomposition of $R_\lambda^0$.  Since $R_\lambda\trianglelefteq G_{I_\lambda^-}$, for every factor $\overline F$ and every $g\in G_{I_\lambda^-}$ we have
\[
R_\lambda\cap g\overline Fg^{-1}
=
g(R_\lambda\cap\overline F)g^{-1}.
\]
We have already shown that, whenever the group in \eqref{eq:R-factor-intersection} is non-trivial, it is either a group in $\Hclass_{h-1}$ or a cyclic group of order $2$. Consequently, the factors in the decomposition of $R_\lambda^0$ supplied by \cref{lem:subgroup-amalgam} are conjugates of groups of these two kinds.

We also have to consider the amalgamated subgroups in this decomposition.  Let $\overline C$ be an amalgamated subgroup of $G_{I_\lambda^-}$.  By the construction of $G_{I_\lambda^-}$, $\overline C$ is either trivial or has order $2$.  Hence every intersection $R_\lambda\cap g\overline Cg^{-1}$ is either trivial or has order $2$.  By~\cref{lem:subgroup-amalgam}, the amalgamated subgroups occurring in $R_\lambda^0$ are contained in groups of this form.  Therefore every amalgamated subgroup of $R_\lambda^0$ is trivial or has order $2$.

We now designate the factors of $R_\lambda^0$ as follows. Those which are conjugate to factors $N_F(I_\lambda^+)/N_F(I_\lambda^-)$ arising from $F\in\Hclass_h$ are designated as factors of type~\textnormal{(i)}. They belong to $\Hclass_{h-1}$.  Those arising from elementary abelian factors of $G$ are designated as factors of type~\textnormal{(ii)}.  Each of these has order $2$ and is therefore an elementary abelian $2$-group.  Since a group of order $2$ has only one non-trivial subgroup, the linear independence condition required for factors of type~\textnormal{(ii)} is automatic. It follows that $R_\lambda^0$ is an admissible amalgam of height at most $h-1$.  By the induction hypothesis, $R_\lambda^0\in\SJ$. The quotient $R_\lambda/R_\lambda^0$ is free and hence belongs to $\SJ$. Consequently, $R_\lambda\in\SJ$.

It remains to refine the complete normal system $\{K_I\}$ in order to obtain a general abelian series for $G$.

For each $\lambda\in\Lambda$, choose a complete general abelian series $\mathcal T_\lambda$ of $ R_\lambda = K_{I_\lambda^+}/K_{I_\lambda^-} $ which witnesses that $R_\lambda\in\SJ$. The inverse images in $K_{I_\lambda^+}$ of the terms of $\mathcal T_\lambda$ form a complete linearly ordered system between $K_{I_\lambda^-}$ and $K_{I_\lambda^+}$.  Thus $\mathcal T_\lambda$ gives a refinement of the interval $K_{I_\lambda^-}\leq K_{I_\lambda^+}$. Moreover, the factors at the jumps of this lifted series are the same as the corresponding factors of $\mathcal T_\lambda$, and are therefore abelian. Clearly very term introduced by these refinements has finite subnormal defect in $G$.

We now insert, for every $\lambda\in\Lambda$, the lifted series $\mathcal T_\lambda$ into the interval $[K_{I_\lambda^-},K_{I_\lambda^+}]$ of the chain $\{K_I\}$. These refinements are compatible with the order on $\Lambda$: if $\lambda<\mu$, then $I_\lambda^+\subseteq I_\mu^-$, and consequently $K_{I_\lambda^+}\leq K_{I_\mu^-}$. Thus all the lifted series, together with the groups~$K_I$, form a single linearly ordered system.

This refined system is complete.  Indeed, consider any non-empty family of its terms.  If a supremum or an infimum lies inside one of the intervals $[K_{I_\lambda^-},K_{I_\lambda^+}]$, it exists because the lifted series $\mathcal T_\lambda$ is complete.  Otherwise the family meets arbitrarily many of the intervals, and its supremum or infimum is a boundary term $K_I$.  The existence of the required boundary term follows from \eqref{eq:KI-union} and \eqref{eq:KI-intersection}, which state that the system $\{K_I\}$ is closed under unions and intersections of chains.

Finally, every jump of the refined system occurs inside one of the lifted series~$\mathcal T_\lambda$.  Indeed, the only possible jumps of the original chain of initial sections are $I_\lambda^-<I_\lambda^+$, and each corresponding interval $K_{I_\lambda^-}\leq K_{I_\lambda^+}$ has now been refined by $\mathcal T_\lambda$.  Since every jump of each~$\mathcal T_\lambda$ has abelian factor, every jump of the resulting complete system has abelian factor.

We have therefore obtained a complete general abelian series of $G$ all of whose terms have finite subnormal defect in $G$.  Hence $G\in\SJ$.
\end{proof}

\section{The counterexample}

Let $\omega_1$ be the first uncountable ordinal.  For every $\beta<\omega_1$ choose an injection \( e_\beta:\beta\longrightarrow\N, \) and, for $\alpha<\beta$, put \( n(\alpha,\beta)=e_\beta(\alpha)+1. \) Let
\[
 E=\bigoplus_{\alpha<\omega_1}\langle g_\alpha\rangle,
 \qquad |g_\alpha|=2.
\]
For each $\alpha<\beta$ take a copy $A_{\alpha\beta}$ of the group $A_{n(\alpha,\beta)}$ supplied by \cref{prop:pairhulls}, with marked commuting involutions $x_{\alpha\beta},y_{\alpha\beta}$. Define $G$ by the presentation obtained from
\[
 E * \mathop{*}_{\alpha<\beta<\omega_1} A_{\alpha\beta}
\]
by imposing the identifications $x_{\alpha\beta}\mapsto g_\alpha$ and $y_{\alpha\beta}\mapsto g_\beta$. For each pair $\alpha<\beta$, these two relations identify the Klein four group $\langle x_{\alpha\beta},y_{\alpha\beta}\rangle$ with the Klein four subgroup $\langle g_\alpha,g_\beta\rangle$ of $E$.  After well-ordering the set of pairs $\alpha<\beta<\omega_1$, we may therefore construct $G$ by adjoining the factors $A_{\alpha\beta}$ one at a time to the initial factor $E$ over these Klein four subgroups, and by taking unions at limit stages.  Thus $G$ is an iterated amalgam in the sense of \cref{sec:amalgams}.  The normal-form theorem at the successor stages, together with the unions at limit stages, shows that $E$ and every $A_{\alpha\beta}$ embed in $G$.

\begin{proposition}\label{prop:notSJ}
The group $G$ does not belong to $\SJ$.
\end{proposition}
\begin{proof}
Suppose that $G$ has an $\SJ$-series $\mathcal S$.  For each $\alpha<\omega_1$, let $A_\alpha<B_\alpha$ be the jump of~$\mathcal S$ containing $g_\alpha$, and put $d_\alpha=\max\{\defect_G(A_\alpha),\defect_G(B_\alpha)\}$. Since $\omega_1$ is uncountable and $\N$ is countable, there are an integer $d$ and an uncountable set $S\subseteq\omega_1$ such that $d_\alpha=d$ for $\alpha\in S$. Choose $\beta\in S$ for which $S\cap\beta$ contains at least $d+1$ elements.  Since $e_\beta$ is injective, some $\alpha\in S\cap\beta$ satisfies $e_\beta(\alpha)\geq d$. Hence $n(\alpha,\beta)=e_\beta(\alpha)+1>d$.

Inside the factor $A_{\alpha\beta}$, \cref{prop:pairhulls} gives that $\langle g_\alpha,g_\beta\rangle^{A_{\alpha\beta},n(\alpha,\beta)}$ is perfect.  By \cref{lem:ambient}, the same iterated normal closure taken in $G$ is perfect. \cref{lem:pair-obstruction} therefore yields $\max\{d_\alpha,d_\beta\}>n(\alpha,\beta)$, whereas $d_\alpha=d_\beta=d<n(\alpha,\beta)$.  This contradiction proves $G\notin\SJ$.
\end{proof}

In order to complete the proof of our main theorem, it remains to show that every countable subgroup of $G$ belongs to $\SJ$.  For a countable set $S\subseteq\omega_1$, put $E_S=\bigoplus_{\alpha\in S}\langle g_\alpha\rangle$ and let $G_S= \left\langle E_S,\ A_{\alpha\beta}: \alpha,\beta\in S,\ \alpha<\beta \right\rangle \leq G$.  .

\begin{lemma}\label{lem:support}
Every countable subgroup of $G$ is contained in $G_S$ for some countable set $S\subseteq\omega_1$.
\end{lemma}

\begin{proof}
Every element of $G$ has a reduced expression involving only finitely many groups~$A_{\alpha\beta}$ and an element of $E$ with finite support.  Hence every element of $G$ belongs to $G_S$ for some finite set $S\subseteq\omega_1$.  If $C\leq G$ is countable, take the union of such finite sets over all elements of $C$.  The resulting set $S$ is countable and $C\leq G_S$.
\end{proof}

We shall prove that $G_S\in\SJ$ for every countable $S$.  Choose an initial segment $J$ of $\mathbb N$ and an enumeration $S=\{s_i:i\in J\}$. Thus either $J=\mathbb N$ or $J=\{0,\ldots,n-1\}$ for some $n$. For $i<j$ in $J$, write $g_i=g_{s_i}$ and $A_{ij}=A_{s_is_j}$.  Let $x_i^{ij}$ and $x_j^{ij}$ be the two distinguished involutions of $A_{ij}$ which are identified with $g_i$ and $g_j$, respectively, and put $X_i^{ij}=\langle x_i^{ij}\rangle^{A_{ij}}$ and $N_{ij}=\langle x_i^{ij},x_j^{ij}\rangle^{A_{ij}}$.

\begin{lemma}\label{lem:marker-intersection}
For every $i<j$ in $J$, $X_i^{ij}\cap\langle x_i^{ij},x_j^{ij}\rangle = \langle x_i^{ij}\rangle$.
\end{lemma}

\begin{proof}
Suppose that $x_j^{ij}\in X_i^{ij}$.  Since $X_i^{ij}$ is normal in $A_{ij}$ and contains both $x_i^{ij}$ and $x_j^{ij}$, it contains their normal closure $N_{ij}$.  Conversely, $N_{ij}$ is normal in $A_{ij}$ and contains $x_i^{ij}$, and therefore $X_i^{ij}\leq N_{ij}$.  Thus $X_i^{ij}=N_{ij}$.  This is impossible, since $N_{ij}$ is perfect by \cref{prop:pairhulls}, whereas $X_i^{ij}$ is not perfect by \cref{lem:cheapmarker}.  Hence $x_j^{ij}\notin X_i^{ij}$.

Since $\langle x_i^{ij},x_j^{ij}\rangle$ is elementary abelian of order $4$, the result follows.
\end{proof}

For $m\geq0$, define
\[
R_m=
\left\langle g_i:i\in J,\ i<m\right\rangle^{G_S},
\qquad
G_m=G_S/R_m.
\]
Thus $R_0=\{1\}$ and $G_0=G_S$.  If $J$ is finite, the groups $R_m$ and $G_m$ become constant once $m$ is larger than every element of $J$.

Fix $m\geq0$.  Put
\[
E_m=
\bigoplus_{\substack{k\in J\\ k\geq m}}
\langle g_k\rangle,
\qquad
A_{ij}^{(m)}=
\begin{cases}
A_{ij}, & m\leq i,\\
A_{ij}/X_i^{ij}, & i<m\leq j,\\
A_{ij}/N_{ij}, & j<m,
\end{cases}
\tag{1}\label{eq:Aijm}
\]
for $i<j$ in $J$.  We first show that these groups, with the surviving prescribed identifications, form an iterated amalgam.

If $m\leq i<j$, neither $g_i$ nor $g_j$ has been killed, and we keep the factor $A_{ij}$ unchanged.  If $i<m\leq j$, the involution $x_i^{ij}$ is to be killed, so we replace $A_{ij}$ by $A_{ij}/X_i^{ij}$.  By \cref{lem:marker-intersection}, $x_j^{ij}\notin X_i^{ij}$, and therefore the image of $x_j^{ij}$ in $A_{ij}/X_i^{ij}$ is still a non-trivial involution.  Hence it may still be identified with the surviving element $g_j\in E_m$.  Finally, if $j<m$, both distinguished involutions are killed, so we replace $A_{ij}$ by $A_{ij}/N_{ij}$; in this case neither of the two original identifications survives.

There is a natural epimorphism $G_S\to\widetilde G_m$.  On $E_S$ it sends $g_i$ to $1$ for $i<m$ and fixes $g_i$ for $i\geq m$.  On a factor $A_{ij}$ it is the identity if $m\leq i$, the quotient map onto $A_{ij}/X_i^{ij}$ if $i<m\leq j$, and the quotient map onto $A_{ij}/N_{ij}$ if $j<m$.  The preceding paragraph shows that these maps respect all the prescribed identifications.  Since every $g_i$ with $i<m$ is killed, this epimorphism kills $R_m$ and therefore induces an epimorphism
\[
G_m=G_S/R_m\longrightarrow\widetilde G_m.
\]

Conversely, let $q_m:G_S\to G_m$ be the quotient map.  If $i<m$, then $q_m(g_i)=1$.  Hence, whenever $i<m\leq j$, the restriction of $q_m$ to $A_{ij}$ kills $x_i^{ij}$ and therefore kills its normal closure $X_i^{ij}$.  It consequently induces a homomorphism $A_{ij}/X_i^{ij}\to G_m$.  If $j<m$, it kills both $x_i^{ij}$ and $x_j^{ij}$, and therefore induces a homomorphism $A_{ij}/N_{ij}\to G_m$.  The surviving elements $g_k$, $k\geq m$, also give a homomorphism $E_m\to G_m$.  These maps satisfy the same prescribed identifications, and hence together define a homomorphism $\widetilde G_m\longrightarrow G_m$.

The two homomorphisms just constructed are inverse to one another, since their compositions are the identity on all the factors which generate the two amalgams.  Thus $G_m\simeq\widetilde G_m$.

By the normal-form theorem at each successor stage, and by taking unions at limit stages, all the factors of $\widetilde G_m$ embed in $\widetilde G_m$. Consequently the image of $E_S$ in $G_m$ is naturally identified with $E_m$, and, for every $i<j$ in $J$, the image of $A_{ij}$ in $G_m$ is naturally isomorphic to the group $A_{ij}^{(m)}$ in~\eqref{eq:Aijm}.

For every $m\geq0$, since $R_m\leq R_{m+1}$ there is a natural epimorphism $G_m\to G_{m+1}$.  Put
\[
K_m=\ker(G_m\longrightarrow G_{m+1})
    =R_{m+1}/R_m.
\]
If $m\notin J$, then, since $J$ is an initial segment of $\mathbb N$, there is no index of $J$ greater than or equal to $m$.  Hence $R_m=R_{m+1}$ and $K_m=\{1\}$.  In studying the non-trivial groups $K_m$ we may therefore assume that $m\in J$.

\smallskip

We now determine the intersections of $K_m$ with the factor groups of $G_m$.  On $E_m$, the map $G_m\to G_{m+1}$ has kernel $\langle g_m\rangle$, and hence $K_m\cap E_m=\langle g_m\rangle$. Consider now a factor corresponding to a pair $i<j$.  If $i=m$, then $j>m$, and the factor $A_{mj}$ in $G_m$ is mapped onto $A_{mj}/X_m^{mj}$.  Hence $K_m\cap A_{mj}=X_m^{mj}$. If $j=m$, then $i<m$, and the factor $A_{im}/X_i^{im}$ in $G_m$ is mapped onto $A_{im}/N_{im}$.  Hence
\[
K_m\cap(A_{im}/X_i^{im})
=
N_{im}/X_i^{im}
=:W_{im}.
\] All remaining factors are unchanged in passing from $G_m$ to
$G_{m+1}$.  Indeed, they correspond to pairs satisfying either
$m<i<j$, or $i<m<j$, or $i<j<m$.  Therefore $K_m$ has trivial
intersection with each of these factors.

Let $J_m$ be the normal subgroup of $K_m$ generated by its
intersections with all conjugates of the factor groups of $G_m$.
Apply \cref{lem:subgroup-amalgam} to $K_m\leq G_m$, so $J_m$ is an iterated amalgam whose factors are of the form $K_m\cap gFg^{-1}$, where $F$ is a factor of $G_m$ and $g\in G_m$. Since $K_m\trianglelefteq G_m$, we have $K_m\cap gFg^{-1}
=
g(K_m\cap F)g^{-1}$. By the preceding computation of the intersections $K_m\cap F$, the
non-trivial factors of $J_m$ are therefore conjugates in $G_m$
of groups of the following three kinds:
\[
\langle g_m\rangle\simeq C_2,\qquad
X_m^{mj}\quad(j>m),\qquad
W_{im}\quad(i<m).
\]
Every non-trivial amalgamated subgroup in this decomposition has order $2$.  Indeed, by \cref{lem:subgroup-amalgam} it is contained in a group of the form $K_m\cap gCg^{-1}$, where $C$ is an amalgamated subgroup of $G_m$.  Since $K_m\trianglelefteq G_m$, this group is $g(K_m\cap C)g^{-1}$.  Every amalgamated subgroup $C$ of $G_m$ is contained in the factor $E_m$, while $K_m\cap E_m=\langle g_m\rangle$.  Hence $K_m\cap C\leq\langle g_m\rangle$, and so every amalgamated subgroup occurring in the decomposition of $J_m$ is trivial or has order $2$.

Consider now a factor of $J_m$ of the form $gX_m^{mj}g^{-1}$, and let $D$ be a non-trivial amalgamated subgroup contained in this factor.  Since $K_m\trianglelefteq G_m$ and $K_m\cap A_{mj}=X_m^{mj}$, we have
\[
K_m\cap gA_{mj}g^{-1}
=
gX_m^{mj}g^{-1}.
\]
Applying \cref{lem:subgroup-amalgam} with $H=K_m$, there are an amalgamated subgroup $C\leq A_{mj}$ of $G_m$ and an element $a\in A_{mj}$ such that $D\leq K_m\cap gaCa^{-1}g^{-1}$. Conjugating by $g^{-1}$ and using $K_m\trianglelefteq G_m$, we obtain $g^{-1}Dg\leq K_m\cap aCa^{-1}$. Since $aCa^{-1}\leq A_{mj}$, it follows that
\[
g^{-1}Dg\leq K_m\cap aCa^{-1}
=
(K_m\cap A_{mj})\cap aCa^{-1}
=
X_m^{mj}\cap aCa^{-1}.
\] In the iterated amalgam decomposition of $G_m$, the only
amalgamated subgroup contained in the factor $A_{mj}$ is $C=\langle x_m^{mj},x_j^{mj}\rangle$,
through which $A_{mj}$ is amalgamated with $E_m$.
Since $X_m^{mj}\trianglelefteq A_{mj}$,
\cref{lem:marker-intersection} gives
\[
\begin{aligned}
X_m^{mj}\cap
a\langle x_m^{mj},x_j^{mj}\rangle a^{-1}
=
a\bigl(
X_m^{mj}\cap
\langle x_m^{mj},x_j^{mj}\rangle
\bigr)a^{-1}=
a\langle x_m^{mj}\rangle a^{-1}.
\end{aligned}
\]
Thus $g^{-1}Dg\leq a\langle x_m^{mj}\rangle a^{-1}$. The group on the right has order $2$, and $D$ is non-trivial. Therefore
\[
g^{-1}Dg
=
a\langle x_m^{mj}\rangle a^{-1}.
\]
Hence, after conjugating $D$ back into $X_m^{mj}$, its non-identity element is an $A_{mj}$-conjugate of~$x_m^{mj}$.

\medskip

\begin{lemma}
$K_m\in\SJ$.
\end{lemma}
\begin{proof}
Fix $j\in J$ with $j>m$, and consider a factor of $J_m$ of the form $gX_m^{mj}g^{-1}$. Let $D$ be a non-trivial amalgamated subgroup contained in this factor. By the preceding discussion, after conjugating by $g^{-1}$ there is an element $a\in A_{mj}$ such that $g^{-1}Dg = a\langle x_m^{mj}\rangle a^{-1}$. Thus, if $y=ax_m^{mj}a^{-1}$, then $g^{-1}Dg=\langle y\rangle$.

By \cref{lem:cheapmarker}, $X_m^{mj}/(X_m^{mj})'$ is an elementary abelian $2$-group.  Moreover, the proof of~\cref{lem:cheapmarker} shows that, under the identification with $W_b$ used there, the image of every~\hbox{$A_{mj}$-con}\-jugate of $x_m^{mj}$ is of the form $v_r+v_{br}$ for some $r\in R$.  By \cref{lem:independent-involutions}, these elements are non-zero, and the distinct ones are linearly independent. Hence $y\notin (X_m^{mj})'$. Since $\langle y\rangle$ has order $2$, it follows that $(X_m^{mj})'\cap g^{-1}Dg=\{1\}$. Conjugating by $g$, we obtain
\[
g(X_m^{mj})'g^{-1}\cap D=\{1\}.
\]
The same conclusion is immediate when $D=\{1\}$.

We may therefore apply \cref{lem:factor-quotient} to $J_m$, factoring each factor conjugate to some~$X_m^{mj}$ by the corresponding conjugate of $(X_m^{mj})'$, and leaving all the other factors unchanged.  This gives an epimorphism $J_m\longrightarrow\overline J_m$. Let $D_m$ be its kernel.  By \cref{lem:factor-quotient}, $D_m$ is a free product of conjugates of the groups $(X_m^{mj})'$, for $j\in J$ with $j>m$, together with a possibly trivial free group.

For every $j\in J$ with $j>m$, $(X_m^{mj})'\leq A_{mj}$. Since $A_{mj}$ has finite normal height, \cref{thm:finite-height} implies that $(X_m^{mj})'$ also has finite normal height.  Hence $(X_m^{mj})'\in\SJ$ by~\cref{lem:heightSJ}.  It follows from \cref{lem:freeproducts} that $D_m\in\SJ$.

The non-trivial factor groups of $\overline J_m$ are now, up to conjugacy,
\[
C_2,\qquad
X_m^{mj}/(X_m^{mj})'\quad(j\in J,\ j>m),\qquad
W_{im}\quad(i\in J,\ i<m).
\tag{4}\label{eq:Jmbar-factors}
\]
The first two kinds are elementary abelian $2$-groups.  Moreover, in each factor $X_m^{mj}/(X_m^{mj})'$, the non-identity elements of the distinct non-trivial amalgamated subgroups are linearly independent; for a factor of order $2$ this condition is immediate.

Remove temporarily all factors conjugate to one of the groups $W_{im}$, $i<m$, together with the amalgamations involving those factors.  The remaining factors split into maximal subfamilies in which any two factors are connected by a sequence of the remaining amalgamations; denote these subfamilies by $\mathcal C_\lambda$. Every factor in $\mathcal C_\lambda$ is elementary abelian, and every amalgamated subgroup occurring between factors of $\mathcal C_\lambda$ is trivial or has order $2$. Clearly, we may first perform all the amalgamations belonging to a fixed component $\mathcal C_\lambda$. Denote the resulting group by $H_\lambda$.  In applying \cref{lem:abelian-compression}, we also specify the subgroups of order $2$ through which the factors belonging to $\mathcal C_\lambda$ are amalgamated with the factors removed above.  The required independence condition holds by the preceding paragraph. We therefore obtain an epimorphism
\[
\pi_\lambda:H_\lambda\longrightarrow L_\lambda
\]
such that $L_\lambda$ is elementary abelian and $\ker\pi_\lambda$ is free.  Moreover, every specified subgroup of order $2$ embeds in $L_\lambda$, and, after repetitions are removed, the non-identity elements of their distinct images are linearly independent.

For each $\lambda$, replace the factor $H_\lambda$ by $L_\lambda$ via the epimorphism $\pi_\lambda:H_\lambda\longrightarrow L_\lambda$, and leave every factor conjugate to some $W_{im}$ unchanged.  If $C\leq H_\lambda$ is an amalgamated subgroup through which $H_\lambda$ is joined to one of the remaining factors, replace $C$ on the $H_\lambda$ side by its image $\pi_\lambda(C)\leq L_\lambda$. All amalgamations between factors conjugate to the groups $W_{im}$ are left unchanged.  Denote by $\widehat J_m$ the iterated amalgam obtained in this way.

This construction is well defined.  Indeed, by the preceding application of \cref{lem:abelian-compression}, every amalgamated subgroup $C\leq H_\lambda$ through which $H_\lambda$ is joined to a remaining factor embeds under $\pi_\lambda$.  Hence $\ker\pi_\lambda\cap C=\{1\}$. For a factor conjugate to some $W_{im}$ we take the trivial normal subgroup.  Therefore the hypotheses of \cref{lem:factor-quotient} are satisfied, with $\ker\pi_\lambda\trianglelefteq H_\lambda$ on the factors $H_\lambda$ and with the trivial subgroup on all the remaining factors.  We obtain a natural epimorphism $\overline J_m\longrightarrow\widehat J_m$.

By \cref{lem:factor-quotient}, its kernel is a free product of conjugates of the groups $\ker\pi_\lambda$, together with a possibly trivial free group.  Since every $\ker\pi_\lambda$ is free, the kernel of $\overline J_m\longrightarrow\widehat J_m$ is free. The factors of $\widehat J_m$ are, up to conjugacy, the groups $L_\lambda$ and the groups $W_{im}$, for $i\in J$ with $i<m$.  Each $L_\lambda$ is an elementary abelian $2$-group.  Every amalgamated subgroup of $\widehat J_m$ is trivial or has order $2$. Moreover, the independence condition holds in every $L_\lambda$. Indeed, in applying \cref{lem:abelian-compression} to~$H_\lambda$, we specified precisely the subgroups of order $2$ through which~$H_\lambda$ is joined to the remaining factors.  The lemma shows that two of their images in $L_\lambda$ are equal only when this equality is forced by a chain of the amalgamations used in forming $H_\lambda$, and that, after equal images are identified, their non-identity elements are linearly independent.  These distinct images are exactly the distinct non-trivial amalgamated subgroups contained in $L_\lambda$ in the decomposition of $\widehat J_m$.

It remains to consider the factors $W_{im}$.  For every $i\in J$ with $i<m$,
\[
X_i^{im}\trianglelefteq N_{im}\leq A_{im},
\qquad
W_{im}=N_{im}/X_i^{im}.
\]
Since $A_{im}$ has finite normal height, \cref{thm:finite-height} implies that $N_{im}$ has finite normal height, and then that its factor group $W_{im}$ has finite normal height.  There are only finitely many indices $i<m$, so there exists $h<\omega$ such that $W_{im}\in\Hclass_h$ for every $i\in J$ with $i<m$. Thus~$\widehat J_m$ satisfies the hypotheses of \cref{thm:amalgam}, and hence $\widehat J_m\in\SJ$.

The kernel of $\overline J_m\to\widehat J_m$ is free and hence belongs to $\SJ$. Since $\SJ$ is closed under extensions, it follows that $\overline J_m\in\SJ$. Again using closure of $\SJ$ under extensions, this gives $J_m\in\SJ$. Finally,
\[
J_m\trianglelefteq K_m
\qquad\text{and}\qquad
K_m/J_m\ \text{is free}.
\]
Therefore $K_m\in\SJ$.
\end{proof}

We can now prove that every $G_S$ is $\SJ$.  Put
\[
R_\omega=\bigcup_{m<\omega}R_m.
\]
Each $R_m$ is normal in $G_S$, and therefore $R_\omega\trianglelefteq G_S$.  Moreover, $R_{m+1}/R_m= K_m\in\SJ$ for every $m$. For each $m$, choose an $\SJ$-series of $R_{m+1}/R_m$ and take its inverse image in $R_{m+1}$.  In this way we obtain a complete general abelian series between $R_m$ and $R_{m+1}$.  Place these series one after another, in increasing order of $m$, and adjoin all the groups $R_m$ and their union $R_\omega$. The resulting chain is complete.  Every jump has abelian factor, and a term lying between $R_m$ and $R_{m+1}$ has finite subnormal defect in $R_{m+1}$, and $R_{m+1}\trianglelefteq R_\omega$. Therefore $R_\omega\in\SJ$.

Now, every $g_i$ is trivial in $G_S/R_\omega$, so  all the prescribed amalgamated subgroups are therefore trivial in this quotient, and the image of $A_{ij}$ is $A_{ij}/N_{ij}$.  Hence
\[
G_S/R_\omega
\simeq
*_{\substack{i,j\in J\\ i<j}} A_{ij}/N_{ij}.
\]
Each group $A_{ij}/N_{ij}$ has finite normal height by \cref{thm:finite-height}, and hence belongs to $\SJ$ by~\cref{lem:heightSJ}.  Therefore $G_S/R_\omega\in\SJ$ by \cref{lem:freeproducts}.  Since $R_\omega\in\SJ$, so $G_S\in\SJ$.

\begin{proposition}\label{prop:countable}
Every countable subgroup of $G$ belongs to $\SJ$.
\end{proposition}
   
\begin{proof}
Let $C\leq G$ be countable.  By \cref{lem:support}, there is a countable set $S\subseteq\omega_1$ such that $C\leq G_S$.  We have just proved that $G_S\in\SJ$.  Hence $C\in\SJ$.
\end{proof}

\section*{Acknowledgements}
The authors are members of the non-profit association ``AGTA -- Advances in Group Theory and Applications'' (www.advgrouptheory.com) and are supported by GNSAGA (INdAM).

During the preparation of this manuscript, the authors used OpenAI's ChatGPT as an interactive tool for checking mathematical arguments, exploring proof strategies, and improving the exposition. The statements, arguments, and proofs presented in the paper remain primarily the work of the authors, with ChatGPT used in a supporting role. The authors take full responsibility for the content.

\end{document}